\documentclass[10pt,a4paper]{amsart}

\usepackage{amsmath}
\usepackage{amssymb}
\usepackage{amsthm}
\usepackage{graphicx,color}
\usepackage{texdraw}

\newtheorem{theorem}{Theorem}[section]
\newtheorem{proposition}[theorem]{Proposition}
\newtheorem{corollary}[theorem]{Corollary}
\newtheorem{lemma}[theorem]{Lemma}

\newtheorem{remark}[theorem]{Remark}

\def\l{\lambda}
\def\N{\mathbb{N}}

\def\R{\mathbb{R}}

\def\eps{\varepsilon}
\def\spc{H^1_{rad}(B)}
\def\cone{\mathcal{C}}

\newcommand{\cC}{{\mathcal C}}

\newcommand{\cH}{{\mathcal H}}

\newcommand{\cL}{{\mathcal L}}

\title[Neumann problems without growth restrictions]{\sc Increasing radial solutions for Neumann problems without growth restrictions}

\author{Denis Bonheure \and Benedetta Noris
\and Tobias Weth}
\address{
  D{\'e}partement de Math{\'e}matique\\
  Universit{\'e} libre de Bruxelles, CP 214\\
  Boulevard du Triomphe, B-1050 Bruxelles, Belgium}
\email{denis.bonheure@ulb.ac.be}
\address{Dipartimento di Matematica e Applicazioni, Universit\`a degli Studi di Milano-Bicocca, via Bicocca degli Arcimboldi 8, 20126 Milano, Italy}
\email{benedetta.noris1@unimib.it}
\thanks{B. Noris is partially supported by MIUR, Project ``Metodi Variazionali ed Equazioni Differenziali Non Lineari".}
\address{Institut f\"ur Mathematik, Goethe-Universit\"at Frankfurt, Robert-Mayer-Str. 10, 60054 Frankfurt, Germany}
\email{weth@math.uni-frankfurt.de}

\begin{document}

\begin{abstract}
We study the existence of positive increasing radial solutions for superlinear Neumann problems in the ball. We do not impose any growth
condition on the nonlinearity at infinity and our assumptions allow for interactions with the spectrum. In our approach we use both topological
and variational arguments, and we overcome the lack of compactness by
considering the cone of nonnegative, nondecreasing radial functions of $H^1(B)$.
\end{abstract}

\maketitle

\section{Introduction}
In this paper we are mainly concerned with the semilinear Neumann problem
\begin{equation}\label{eq:main_equation}
\left\{\begin{array}{ll}
        -\Delta u+u =a(|x|)f(u) \quad &\text{ in } B \\
	u>0 &\text{ in } B\\
	\partial_\nu u=0 &\text{ on } \partial B,
       \end{array}
\right.
\end{equation}
where $B$ is the unit ball in $\R^N$, $N\geq2$. We study the
existence of radial solutions of \eqref{eq:main_equation} under
 suitable assumptions on $a$ and $f$. The problem has been studied
extensively in the case where $f(u)=u^p$ with some $p>1$ and $a \equiv
1$. Note that in this case there always exists the constant solution
$u \equiv 1$ of (\ref{eq:main_equation}). This already shows that   
the solvability of (\ref{eq:main_equation}) depends in a quite
different way on the data than in 
the case of Dirichlet boundary conditions, in which nontrivial
solutions only exist in the subcritical range 
\begin{equation}
  \label{eq:6}
p<\frac{N+2}{N-2} \qquad \text{if $N \ge 3$}  
\end{equation}
as a consequence of Pohozaev's identity, see \cite{Pohozaev}. Note that the subcriticality assumption (\ref{eq:6}) ensures that the
problem (\ref{eq:main_equation}) with $f(u)=u^p$ is 
accessible by variational methods, i.e., the (formal) energy functional
corresponding to (\ref{eq:main_equation}) is well defined in
$H^1(B)$. Moreover, due to the compact embedding $H^1(B) \hookrightarrow
L^{p+1}(B)$, the existence of a solution to (\ref{eq:main_equation})
follows in a standard way through the mountain pass theorem \cite{ambrosetti-rabinowitz:73}
if $a$ is a positive continuous function on $\overline B$. In the
critical and supercritical case, namely when (\ref{eq:6}) does not hold, most
of the available results on the existence of positive solutions are
devoted to perturbative cases where either a small diffusion constant is added in front
of $-\Delta$, see \cite[Chapter 9 and 10]{Ambrosetti-Malchiodi} and the references therein or a slightly supercritical exponent is considered, see e.g. \cite{delpinomussopistoia}.
The present paper deals with the
nonperturbative problem and is therefore more closely related
to the recent works \cite{Barutello-Secchi-Serra,SerraTilli}. In \cite{Barutello-Secchi-Serra}, the authors
considered the Neumann problem for the
H\'enon equation $-\Delta u+ u = |x|^{\alpha} u^p$, and they apply a
shooting method to prove that
this problem admits a positive and radially increasing solution for
every $p>1$ and $\alpha>0$. Very recently, Serra and Tilli
\cite{SerraTilli} showed the existence of the same type of solutions
for problem (\ref{eq:main_equation}), provided that $a$ is an
increasing positive function with $a(0)>0$ and $f \in \cC^1([0,\infty))$ is such
that 
\begin{align}
  \label{eq:7}
f(0)=f'(0)=0,\quad f'(t)t-f(t)>0 \text{ and } f(t)t \ge \mu F(t):= \int_0^t
f(s)\,ds, 
\end{align}
for $t\in (0,\infty)$, with some constant $\mu>2$. These assumptions, which hold for $f(u)= u^p$, $p>1$, play a crucial role in the approach of
Serra and Tilli, who minimize the energy functional corresponding to
(\ref{eq:main_equation}) among nonnegative, radial and
radially nondecreasing functions within the associated Nehari
manifold. Reducing to nonnegative and nondecreasing radial trial
functions in $H^1(B)$ gives rise to boundedness and compactness properties even for
supercritically growing nonlinearities. It is not
obvious that restrictions of this type still lead to a solution of
(\ref{eq:main_equation}), but Serra and Tilli could prove this with
the help of assumptions (\ref{eq:7}).

The purpose of the present paper is twofold. First, we generalize the
results of Serra and Tilli to a wider class of functions $f$ by means
of a new approach based on topological fixed point theory and invariance properties of the cone
of nonnegative, nondecreasing radial functions in $H^1(B)$. In
particular, we give a rather short proof of the existence
of an increasing radial solution of (\ref{eq:main_equation}). More precisely, we first establish a priori estimates on
the solutions of \eqref{eq:main_equation} in this cone and then apply
a suitable version of Krasnosel$'$ski{\u\i}'s fixed point theorem (see \cite{Kwong}). The second aim of this paper is related to the case $a$
constant, say $a \equiv 1$, where any
fixed point of $f$ gives rise to a constant solution of
(\ref{eq:main_equation}). In this case we will be concerned with the existence of
{\em nonconstant} increasing solutions. 
To state our main results, we now list our assumptions on $a$ and $f$:
\begin{itemize}
\item[($a$)] $a \in C^1([0,1],\R)$ is nondecreasing and
  $a_0:=a(0)>0$;
\item[($f1$)] $f\in C^1([0,+\infty),\R)$, $f(0)=0$ and $\displaystyle{ f'(0)=\lim_{s\to0^+} \frac{f(s)}{s}=0 }$;
\item[($f2$)] $f$ is nondecreasing;
\item[($f3$)] $\displaystyle{  \liminf \limits_{s\to+\infty} \frac{f(s)}{s}>\frac{1}{a_0}  }$.
\end{itemize}
In particular, these assumptions on $f$ allow the nonlinearity to have
supercritical growth as well as resonant growth, i.e. $\lim_{s\to+\infty}
  f(s)/s = \lambda$ with $\lambda>1$ being a Neumann eigenvalue
  of the operator $-\Delta + 1$ in $B$, and they are much weaker than
  (\ref{eq:7}). In particular, $f$ may have multiple positive fixed
  points and the quotient $f(s)/s$ may 
  oscillate between values in an interval of the form
  $[c,\infty)$ with $c>1/a_0$ for large $s$, whereas
  (\ref{eq:7}) forces 
  this quotient to be strictly increasing. Our first existence result for
  (\ref{eq:main_equation}) is the following.
\begin{theorem}\label{theorem:non_constant_a}
Assume ($a$), ($f1$), ($f2$), ($f3$) and suppose moreover that
$a(|x|)$ is nonconstant. Then there exists at least one nonconstant
nondecreasing radial
solution of \eqref{eq:main_equation}.
\end{theorem}

The existence of solutions for such general
nonlinearities $f$ underscore the difference between Dirichlet and
Neumann boundary conditions for supercritical elliptic problems, 
see also the related recent papers
\cite{BonheureSerra,GrossiNoris,Secchi}. In contrast to the method of Serra and Tilli in \cite{SerraTilli}, 
our approach based on topological fixed point theory does not require the (formal) variational
structure of problem (\ref{eq:main_equation}) and therefore applies
to the more general problem
\begin{equation}\label{eq:main_equation-variant}
\left\{\begin{array}{ll}
        -\Delta u+b(|x|)\,  x \!\cdot\! \nabla u  + u =a(|x|)f(u) \quad &\text{ in } B \\
	u>0 &\text{ in } B\\
	\partial_\nu u=0 &\text{ on } \partial B,
       \end{array}
\right.
\end{equation}
provided that the following assumption holds:
\begin{itemize} 
\item[($b$)] $b \in C([0,1],\R)$ is nonpositive, and $\displaystyle{  \frac{d}{dr}
  (b(r)r) > -1-\frac{N-1}{r^2}  }$ in $(0,1)$.  
\end{itemize}
\begin{theorem}\label{theorem:non_constant_a-1}
Assume ($a$), ($b$), ($f1$), ($f2$), ($f3$) and suppose moreover that
$a(|x|)$ is nonconstant. Then there exists at least one nonconstant nondecreasing radial
solution of \eqref{eq:main_equation-variant}.
\end{theorem}

In case $a$ is a constant function, say $a\equiv1$, assumptions
$(f1)$--$(f3)$ imply the existence of $u_0 >0$ such that $f(u_0)=u_0$, 
so that $u_0$ is a constant solution of
(\ref{eq:main_equation}). Moreover, there exist
nonlinearities satisfying $(f1)$--$(f3)$ (with $a_0=1$) and such that the problem
\begin{equation}\label{eq:main_equation_a=1}
\left\{\begin{array}{ll}
        -\Delta u+u=f(u) \quad &\text{ in } B \\
	u>0 &\text{ in } B\\
	\partial_\nu u=0 &\text{ on } \partial B
       \end{array}
\right.
\end{equation}
only admits this constant solution (see Proposition
\ref{proposition:example} below, where we adapt an argument of
\cite{BonheureGrumiau}). We need
the following additional assumption:
\begin{itemize}
\item[($f4$)] there exists $u_0>0$ such that $f(u_0)=u_0$ and $f'(u_0)>\lambda_2^{rad}$.
\end{itemize}
Here $\lambda_2^{rad}>1$ is the second radial eigenvalue of
$-\Delta+1$ in the unit ball with Neumann boundary conditions. We prove the following result.
\begin{theorem}\label{theorem:a=1}
Assume ($f1$)--($f4$) with $a \equiv 1$. Then there exists at least
one nonconstant increasing radial solution of \eqref{eq:main_equation_a=1}.
\end{theorem}

To our knowledge, this is the first existence result for nonconstant
solutions of (\ref{eq:main_equation_a=1}) under assumptions
$(f1)$--$(f4)$ and even under the more restrictive
conditions~(\ref{eq:7}) and $(f4)$. 
An inspection of the proof of Theorem \ref{theorem:a=1} shows that we
find nonconstant solutions of \eqref{eq:main_equation_a=1} in every order interval of the form
$[u_-,u_+]$, where $u_-$ and $u_+$ are ordered fixed points of $f$ with the property that 
there exists another fixed point $u_0 \in (u_-,u_+)$ such that $f'(u_0)>\lambda_2^{rad}$.

We note that the topological fixed point method does not give
sufficient information to detect a nonconstant solution of
(\ref{eq:main_equation_a=1}), moreover it seems impossible to use the
spectral assumption ($f4$) within a shooting approach to derive
Theorem~\ref{theorem:a=1}. Therefore we use a
variational approach, but this leads to several difficulties. First, the (formal) energy functional associated
with (\ref{eq:main_equation_a=1}) is not well defined
and of class $C^1$ in $H^1(B)$ under the sole assumptions
$(f1)$--$(f4)$. We overcome this problem by truncating the
nonlinearity $f$ and by recovering the original problem by means of a
priori estimates on the solutions. Then we construct a suitable convex
subset $\cC_*$ of the cone of
nonnegative, nondecreasing radial functions in $H^1(B)$ such that
$u_0$ is the only constant solution of (\ref{eq:main_equation_a=1}) in
$\cC_*$, and we show
that this set is positively invariant under the corresponding gradient
flow. Then we set up a variational principle of mountain pass type
within $\cC_*$, and -- using assumption $(f4)$ -- we show that the corresponding
critical point is different from $u_0$. 
Within this last step, a further problem
occurs; the set $\cC_*$ has empty interior in the
$H^1$-topology, and even though one could prove that $\cC_* \cap X$ 
has interior points in the topology of the smaller space
$X=\cC^2(\overline B) \subset H^1(\Omega)$, the constant solution
$u_0$ is still a boundary point of $\cC_* \cap X$. Therefore it does
not seem possible to use standard Morse theory (i.e. the calculation
of critical groups) to distinguish critical points obtained via
deformations in $\cC_*$ from the constant solution $u_0$. In particular,
this prevents us from using the techniques in
\cite{Li-Wang}, where the authors prove an abstract mountain pass
theorem in order intervals.

The paper is organized as follows. In Section \ref{section:preliminaries} we introduce the cone of radial, nonnegative, nondecreasing
functions and its properties. In Section \ref{section:topological_method} we obtain a priori estimate on the solutions of \eqref{eq:main_equation}
in the cone, which allows to prove Theorem \ref{theorem:non_constant_a} by applying a suitable fixed point theorem in the cone. In Section
\ref{section:variational_method} we fix $a(|x|)=1$ and provide the
proof of Theorem~\ref{theorem:a=1}. 

We close the introduction with an open problem. Our
construction of the nonconstant solution $u$
of~(\ref{eq:main_equation_a=1}) provided in Theorem~\ref{theorem:a=1}
implies that $u$ intersects the constant solution $u_0$. This raises the question whether it is
possible to construct radial solutions with a given number of
intersections with $u_0$ provided that $f'(u_0)$ is sufficiently
large. More precisely, we conjecture that there 
exists a radial solution with $k$ 
intersections with $u_0$ provided that $f'(u_0) > \lambda_k^{rad}$.

\section{The cone of nonnegative, nondecreasing, radial functions}\label{section:preliminaries}
We will look for solutions to \eqref{eq:main_equation} and (\ref{eq:main_equation-variant}) in the space of radial $H^1$ functions in the ball, that we denote by $\spc$.
If $u\in \spc$ then we can assume it is continuous in $(0,1]$ and the following set is well defined
\[
\cone=\{u\in \spc:\ u\geq0 \text{ and } u(r)\leq u(s) \text{ for every } 0 < r\leq s \leq 1\}.
\]
Observe that if $u\in \cone$, then $u\in C(\overline{B})$, and in particular it is a bounded function. In fact, since $u$ is nondecreasing, we can assume continuity also at the origin by setting $u(0)=\lim_{r\to0^+}u(r)$.
Moreover, $u$ is differentiable almost everywhere and $u'(r)\geq0$ where it is defined.
%

It is easy to see that $\cone$ is a closed convex cone in $H^1(B)$, that is
\begin{itemize}
\item[(i)] if $u\in \cone$ and $\l>0$ then $\l u\in \cone$;
\item[(ii)] if $u,v\in \cone$ then $u+v\in \cone$;
\item[(iii)] if $u, -u\in\cone$ then $u\equiv0$;
\item[(iv)] $\cone$ is closed for the topology of $H^1$.
\end{itemize}
We will refer to $\cone$ as the cone of nonnegative, nondecreasing functions. Notice also that it is weakly closed in $H^1$ and as already mentioned, it has empty interior in the $H^1$-topology.

As observed by Serra and Tilli in \cite{SerraTilli}, $\cone$ is a good set when dealing with supercritical equations because of the a priori bound stated in the following lemma. 
\begin{lemma}\label{lemma:embedding_of_cone_in_L_infty}
There exists a constant $C$ only depending on the dimension $N$ such that
\[
\|u\|_{L^\infty(B)}\leq C \|u\|_{W^{1,1}(B)} \qquad \text{for all } u\in\cone.
\]
\end{lemma}
\begin{proof}
For every $u\in\cone$ we have $\|u\|_{L^\infty(B)}=\|u\|_{L^\infty(B\setminus B_{1/2})}$. Since $u$ is radial and the space
$W^{1,1}((1/2,1))$ is continuously embedded in $L^\infty((1/2,1))$, we deduce that there exists $C>0$, only depending on the dimension $N$,
such that
\[
\|u\|_{L^\infty(B)}=\|u\|_{L^\infty(B\setminus B_{1/2})}\leq C \|u\|_{W^{1,1}(B\setminus B_{1/2}))}\leq C \|u\|_{W^{1,1}(B)}. \qedhere
\]
\end{proof}

\begin{remark}
\label{sec:cone-nonn-nond-1}
Lemma~\ref{lemma:embedding_of_cone_in_L_infty} implies that the
embedding $\cone \subset L^\infty(\Omega)$ is bounded when
$\cone$ is considered with the metric induced by the $H^1(B)$-norm. 
However, this embedding is not continuous if $N \ge 3$, since the sequence $(u_n)_n
\subset \cC$ defined by $u_n(x)=|x|^{1/n}$ satisfies 
$\|u_n-1\|_{H^1(B)} \to 0$ as $n\to+\infty$ and 
$\|u_n-1\|_{L^\infty} \ge 1$ for all $n$. Nevertheless we have the
following continuity property.
\end{remark}

\begin{lemma}
\label{sec:cone-nonn-nond}
Let $g: [0,\infty) \to \R$ be continuous, and let $(u_n)_n \subset
\cone$ be a sequence with $u_n \rightharpoonup u$ weakly in $H^1(B)$. Then for every $p
\in [1,\infty)$ we have
$$
g \circ u_n \to g \circ u \qquad \text{in $L^p(B)$ as $n \to \infty$.}
$$
\end{lemma}

\begin{proof}
Let $p \in [1,\infty)$. Suppose by contradiction that -- passing to a subsequence -- we have 
\begin{equation}
  \label{eq:9}
\liminf_{n \to \infty} \int_{B} |g(u_n) -g(u)|^p \,dx >0.
\end{equation}
Since $u_n \to u$ in $L^2(B)$, we may pass to a subsequence such that
$u_n \to u$ a.e. in $B$. Moreover, by Lemma~\ref{lemma:embedding_of_cone_in_L_infty} we have $u \in L^\infty(B)$
and $\sup \limits_{n \in \N} \|u_n\|_{L^\infty(B)}<\infty$, hence also 
$$
\sup_{n \in \N} \|g(u_n)-g(u)\|_{L^\infty(B)} < \infty.
$$
We now infer from Lebesgue's theorem that 
$$\lim \limits_{n \to \infty} \int_{B} |g(u_n) -g(u)|^p \,dx =0$$ 
but this contradicts (\ref{eq:9}). The claim follows. 
\end{proof}

\section{Existence of solutions via a topological method}\label{section:topological_method}
In this section we will prove Theorem \ref{theorem:non_constant_a-1},
and we note that Theorem~\ref{theorem:non_constant_a} immediately
follows from Theorem \ref{theorem:non_constant_a-1}. Throughout this
section we assume conditions ($a$), ($b$) and ($f1$)--($f3$). We first
recall well known properties of the linear differential operator $\cL:=
-\Delta +b(|x|)\,  x \!\cdot\! \nabla  +Id$. 

\begin{lemma}
\label{sec:exist-solut-via} 
Let 
$$
\cH(B):= \{v \in H^2(B)\,:\, \partial_r v \in H^1_0(B)\},
$$
where $\partial_r$ denotes the derivative in direction $x/|x|$.
For every $w \in L^2(B)$, the equation $\cL v=w$ admits a unique
solution $v \in \cH(B)$, and $\|v\|_{H^2(B)} \le C\|w\|_{L^2(B)}$ with a constant
$C>0$ independent of $w$. 
Moreover, if $w \in L^p(B)$ for some $p \in (2,\infty)$, then $v \in
W^{2,p}(B)$. Also, if $w \in H^1(B)$, then $v \in H^3(B)$.
\end{lemma}

\begin{proof}
The assertions are true by
standard elliptic regularity if $b \equiv 0$. Moreover, since the first
order term in $\cL$ defines a compact perturbation, $\cL$ is
a Fredholm operator of index zero when considered as a map
between the spaces $\cH(B) \to L^2(B)$, $\cH(B) \cap W^{2,p}(B) \to L^p(B)$
and $\cH(B) \cap  H^3(B) \to H^1(B)$, respectively. Therefore it remains to prove
the following: 
\begin{equation}
  \label{eq:injectivity}
\text{the equation $\cL v= 0$ only admits the trivial solution in $\cH(B)$.} 
\end{equation}
To prove this, let $v \in \cH(B)$ solve $\cL v= 0$, i.e. 
$-\Delta v +v= b(|x|)\,  x \!\cdot\! \nabla v$. Since the map $x \mapsto b(|x|)$ is Lipschitz in $B$ as a
consequence of assumption $(b)$, it follows from standard elliptic
regularity that $v \in C^{2,\alpha}(\overline B)$ for some $\alpha>0$. Moreover, by the
strong maximum principle, $v$ neither may attain a positive maximum
nor a negative minimum in $B$. Since
moreover $\partial_r v = 0$ on
$\partial B$, the Hopf Lemma implies that $v$ cannot attain a positive maximum
nor a negative minimum on $\partial B$. Therefore $v \equiv 0$, as
claimed in (\ref{eq:injectivity}).
\end{proof}
   
We will prove Theorem \ref{theorem:non_constant_a-1} by applying a suitable fixed point theorem to the operator $T:\cone \to H^1(B)$
defined as
\begin{equation}\label{eq:definition_of_T}
T(u)=v \qquad \text{with} \qquad
\left\{\begin{array}{ll}
        -\Delta v+b(|x|)\,  x \!\cdot\! \nabla v + v=a(|x|)f(u) \quad &\text{ in } B \\
	\partial_\nu v=0 &\text{ on } \partial B.
       \end{array}
\right.
\end{equation}
Notice that the function $x \mapsto a(|x|)f(u(x))$ is contained in
$\cone$ whenever $u \in \cone$, since $u\in L^\infty(B)$ by Lemma
\ref{lemma:embedding_of_cone_in_L_infty}. The first step is of course to prove
that $T(\cone)\subseteq \cone$.
\begin{lemma}\label{lemma:T_is_cone_map}
Let $w\in \cone$; then the equation 
\[
\left\{\begin{array}{ll}
        -\Delta v+b(|x|)\,  x \!\cdot\! \nabla v +v=w \quad &\text{ in } B \\
	\partial_\nu v=0 &\text{ on } \partial B,
       \end{array}
\right.
\]
admits a unique solution $v=T(w)$, which belongs to $\cone$.
\end{lemma}

\begin{proof}
Since $w \in \cone \subset H^1(B) \cap L^\infty(B)$, it follows from
Lemma \ref{sec:exist-solut-via} that there exists
a unique solution $v$ in $\cH(B) \cap H^3(B) \cap W^{2,p}(B)$ (for every
$p<\infty$). Hence $v \in C^{1,\alpha}(\overline B)$ and $\partial_\nu v= 0$ on $\partial B$. Since the solution is radial (because
it is unique), we may write the equation for $v$ in polar coordinates
as 
\[
-v''+ \Bigl(b(r)r -\frac{N-1}{r}\Bigr) v'+v=w, \qquad v'(0)=v'(1)=0,
\]
where $v'$ denotes the derivative with respect to $r=|x|$. Note that,
as a function of $r$, we have $z:=v' \in H^2_{loc}(0,1)$, so
differentiation yields
\[
\Bigl(b(r)r -\frac{N-1}{r}\Bigr)z'+
\Bigl([b(r)r]'+\frac{N-1}{r^2}+ 1 \Bigr)z =
z'' + w'.
\]
We point out that the left hand side of this equation is continuous in
$(0,1)$ (since $H^2_{loc}(0,1) \subset C^1(0,1)$), and thus the continuity of the right hand side follows. Now suppose by
contradiction that $z$ attains a negative local minimum at a
point $r_0 \in (0,1)$, then at this point we have $z'(r_0)=0$ and 
$$
\Bigl([b(r)r]'+\frac{N-1}{r^2}+ 1 \Bigr)z\Big|_{r_0}<0
$$
by assumption ($b$). Therefore, by continuity, there exists a neighborhood $U$ of $r_0$ in
$(0,1)$ with
$$
\Bigl(b(r)r -\frac{N-1}{r}\Bigr)z'+
\Bigl([b(r)r]'+\frac{N-1}{r^2}+ 1 \Bigr)z <0 \qquad \text{in $U$.}
$$
Since $w' \ge 0$ in $(0,1)$, it then follows that 
$z''<0$ a.e. in $U$, which yields that $z'$ is strictly decreasing in $U$. This
however contradicts our assumption that $z$ attains a negative minimum
at $r_0$. Since moreover $z(0)=z(1)=0$, we conclude that $v'=z \ge
0$ in $(0,1)$, so that $v \in \cone$.  
\end{proof}

\begin{corollary}\label{corollary:T_is_cone_map}
The operator $T$ defined by \eqref{eq:definition_of_T} satisfies $T(\cone)\subseteq \cone$.
\end{corollary}
\begin{proof}
Observe that if $u\in \cone$, the assumptions on $a(r)$ and $f$ imply that $a(r)f(u)\in \cone$. Henceforth, the conclusion follows from Lemma
\ref{lemma:T_is_cone_map}.
\end{proof}

In order to apply a fixed point theorem in the cone, we need a priori estimates on the solutions of \eqref{eq:main_equation} and on the
solutions of a family of auxiliary problems depending on some parameters $\lambda\geq0$ and $0<\mu\leq1$.
\begin{lemma}\label{lemma:K_1}
There exists a constant $\bar \lambda$ such that the following problem
\begin{equation}\label{eq:equation_with_lambda}
\left\{\begin{array}{ll}
        -\Delta u+b(r)\, x\!\cdot\! \nabla u + u=a(r)f(u)+\l \quad &\text{ in } B \\
	u\geq0 &\text{ in } B\\
	\partial_\nu u=0 &\text{ on } \partial B,
       \end{array}
\right.
\end{equation}
does not admit any solution in $\cC$ for $\l>\bar\lambda$. Moreover, there exists a constant $K_1$ such that every solution $u$ of
\eqref{eq:equation_with_lambda} with $0\leq\l\leq \bar\lambda$ satisfies $\|u\|_{L^1(B)}\leq K_1$.
\end{lemma}
\begin{proof}
By assumption ($f3$) there exists $M,\delta >0$ such that
\begin{equation}
  \label{eq:M-delta-estimate}
\frac{f(s)}{s}\geq \frac{1+\delta}{a_0} \qquad \text{for every } s\geq M,
\end{equation}
where $a_0=a(0)$. Let $u\in\cone$ be a solution of \eqref{eq:equation_with_lambda}. Since $b(r)x\!\cdot\! \nabla u(x) \le 0$ by assumption ($b$), integrating the equation in \eqref{eq:equation_with_lambda} in $B$ yields
\begin{align*}
\int_B u\,dx &\ge  \int_B [u+ b(r)\, x\!\cdot\! \nabla u(x)
]\,dx=\int_{\{u\leq M\}} a(r) f(u)\,dx+\int_{\{u> M\}} a(r) f(u)\,dx
+\l |B|\\
& \geq  \int_{\{u> M\}} a(r) \frac{1+\delta}{a_0} u \,dx +\l |B| \ge 
(1+\delta) \int_{\{u> M\}} u \,dx +\l |B|.
\end{align*}
Therefore
\[
M|B| \geq \int_{\{u\leq M\}} u \,dx \geq \delta \int_{\{u> M\}} u \,dx +\l |B|
\]
and the lemma is proved.
\end{proof}

From now on, we fix $\bar \l$ as in the previous lemma.
\begin{lemma}\label{lemma:K_infty}
Assume $0\leq\l\leq\bar\l$. There exist two constants $K_\infty, K_2$ such that if $u\in\cone$ solves \eqref{eq:equation_with_lambda}, then 
\[
\|u\|_{L^\infty(B)}\leq K_\infty \qquad \text{and}\qquad \|u\|_{H^1(B)}\leq K_2.
\]
\end{lemma}
\begin{proof}
Let $u\in\cone$ be a solution of \eqref{eq:equation_with_lambda}. In radial coordinates, the equation for $u$
can be written in the form
\[
(r^{N-1}u')'= r^{N-1} (u(r)+b(r)ru'(r)- a(r)f(u(r))-\l) \le r^{N-1}u(r).
\]
Therefore
\[
u'(r)\leq\frac{1}{r^{N-1}}\int_0^r u(t)t^{N-1}\,dt
\leq\frac{1}{r^{N-1}|\partial B|}\int_B u \,dx\leq\frac{K_1}{r^{N-1} |\partial B|},
\]
with $K_1$ defined in the previous lemma. Since $u'\geq0$, we deduce from the previous inequality that $\|u\|_{W^{1,1}(B)}\leq 2K_1$,
so that Lemma \ref{lemma:embedding_of_cone_in_L_infty} gives the first estimate.
As for the estimate of the $H^1$-norm, we multiply the equation in \eqref{eq:equation_with_lambda} by $u$ and integrating in the ball yields
$$
\int_B\left(|\nabla u|^2+u^2\right)\, dx= \int_B [a(r)f(u)u-b(r)\,
x\!\cdot\! \nabla u]u\,dx +\bar\l \int_B u \,dx.
$$
Since $u$ is a priori bounded in $W^{1,1}(B)$ and $L^\infty(B)$, the
right hand side is a priori bounded as well, and the a priori bound in $H^1(B)$ follows.
\end{proof}

\begin{remark}
\label{sec:exist-solut-via-1}
An inspection of the proofs of Lemma~$\ref{lemma:K_1}$ and
$\ref{lemma:K_infty}$ shows the following. First, it is possible to choose
$$
\bar \lambda:= \min \{s \ge 0\::\: \text{$f(t) \ge t$ for $t \ge s$}\}
$$ 
in Lemma~$\ref{lemma:K_1}$. Moreover, the a priori bounds in
these lemmas only depend on some properties of $f$ and not on the nonlinearity
itself. More precisely, if $M>0$ and $\delta>0$ are fixed, then
$K_1,K_2$ and $K_\infty$ can be chosen independently for all
nonnegative nonlinearities $f$ satisfying
(\ref{eq:M-delta-estimate}). This will be important in Section~$\ref{section:variational_method}$ where we work with a truncated problem.   
\end{remark}
\begin{lemma}\label{lemma:k_infty}
There exists a constant $k_2$ such that for every $0< \mu< 1$ and for every solution $u\not\equiv0$ of
\begin{equation}\label{eq:equation_with_mu}
\left\{\begin{array}{ll}
        -\Delta u+b(r)\, x\!\cdot\! \nabla u+u=\mu a(r)f(u) \quad &\text{ in } B \\
	u\geq0 &\text{ in } B\\
	\partial_\nu u=0 &\text{ on } \partial B,
       \end{array}
\right.
\end{equation}
we have $\|u\|_{H^1(B)}\geq k_2$.
\end{lemma}
\begin{proof}
By Lemma~\ref{sec:exist-solut-via}, there exists a constant $C>0$ such that 
\begin{equation}
  \label{eq:isomorphism}
\|-\Delta u +b(r)\,  x \!\cdot\! \nabla u +u\|_{L^2(B)} \ge C
\|u\|_{L^2(B)}
\qquad \text{for all $u \in \cH(B)$.}
\end{equation}
Assume by contradiction the existence of $u_n\not\equiv0$, solutions of \eqref{eq:equation_with_lambda} with $0<\mu_n<1$, such that
$\|u_n\|_{H^1(B)}\to0$ as $n\to+\infty$. Then
$\|u_n\|_{L^\infty(B)}\to0$ by Lemma
\ref{lemma:embedding_of_cone_in_L_infty}. By assumption ($f1$) we have
\[
\frac{f(u_n(x))}{u_n(x)} \leq \frac{1}{n} \quad \text{for all $x \in B$},
\]
for $n$ sufficiently large, and it then follows from
(\ref{eq:isomorphism}) that 
\[
C^2 \|u_n\|_{L^2(B)}^2 \le \mu_n^2 \int_B [a(r) f(u_n)]^2 
\leq \Bigl(\frac{\mu_n a(1)}{n}\Bigr)^2 \int_B u_n^2\, dx
=\Bigl(\frac{\mu_n a(1)}{n}\Bigr)^2 \|u_n\|_{L^2(B)}^2.
\]
Since $u_n\not\equiv0$ for every $n$, this yields a contradiction for $n$ large.
\end{proof}

We now turn to the proof of Theorem \ref{theorem:non_constant_a}. We are in a position to apply a generalization of a fixed point theorem by Krasnosel$'$ski{\u\i} (see \cite{Krasnoselskii2,Krasnoselskii1}) to the
operator $T$ defined by \eqref{eq:definition_of_T} in the cone $\cone$.
This theorem is proved by Benjamin in \cite{Benjamin}, Appendix 1, but we refer to Kwong \cite{Kwong} where the approach is more elementary. We also quote \cite{Amann} and \cite{Nussbaum}.

\begin{proof}[Proof of Theorem \ref{theorem:non_constant_a}]
Let us check the assumptions of the fixed point theorem in \cite{Kwong} (expansive form) :
\begin{itemize}
\item[(i)] $T:\cone\to\cone$ by Corollary \ref{corollary:T_is_cone_map} ;
\item[(ii)]  $T$ is completely continuous on $\cone$. Indeed let $\{u_n\}\subset \cone$ be a sequence bounded in $H^1(B)$. By Lemma
\ref{lemma:embedding_of_cone_in_L_infty} it is bounded in
$L^\infty(B)$, hence $\{v_n=T(u_n)\}$ is
bounded in $H^2(B)$ by Lemma~\ref{sec:exist-solut-via}. Therefore, by
the compactness of the embedding $H^2(B) \hookrightarrow H^1(B)$,
a subsequence of $(v_n)_n$ converges in the $H^1$-norm ;
\item[(iii)] For every $\l\geq0$, for every $u\in \cone$ with
  $\|u\|_{H^1(B)}=2K_2$ (defined in Lemma \ref{lemma:K_infty}) we have $u-T(u)\neq\l$. In fact notice that $u-T(u)=\l$ if and only if $u$ solves equation \eqref{eq:equation_with_lambda}, hence
this property is a consequence of Lemma \ref{lemma:K_infty} ;
\item[(iv)] for every $0<\mu<1$, for every $u\in \cone$ with
  $\|u\|_{H^1(B)}=k_2/2$ (defined in Lemma \ref{lemma:k_infty})
  we have $\mu T(u)\neq u$. In fact we have $\mu T(u)=u$ if and only if $u$ solves equation \eqref{eq:equation_with_mu}, hence
property (iv) is a consequence of Lemma \ref{lemma:k_infty}.
\end{itemize}
We then conclude that there exists a fixed point of $T$ in $\cone$. Such a fixed point is of course a nonconstant solution of \eqref{eq:main_equation} since $a$ is nonconstant. Moreover it is strictly positive and strictly increasing by the maximum principle. This completes the proof.
\end{proof}

\section{Existence of solutions via a variational method}\label{section:variational_method}
In the case where $a$ is a constant function, say $a\equiv1$, the following
proposition and remark show that (\ref{eq:main_equation_a=1}) may only admit the
constant solution $u \equiv u_{0}$ in $\spc$. The argument is adapted from \cite{BonheureGrumiau} where it is shown that if $f(u)=u^{p}$ and $p$ is close to $1$, $u_{0}\equiv 1$ is the unique solution of (\ref{eq:main_equation_a=1}). 

Recall that $\lambda_2^{rad}>1$ is the second radial eigenvalue of
$-\Delta+1$ in the unit ball with Neumann boundary conditions.
Fix $\delta \in (0,\lambda_2^{rad})$ and let $M>0$.  By Lemma~\ref{lemma:K_infty} and Remark \ref{sec:exist-solut-via-1}, there
exists $K_\infty>0$ such that, if $f$ satisfies $(f1)$--$(f3)$ and 
(\ref{eq:M-delta-estimate}) with these values of $M$, $\delta$ and $a_0 \equiv 1$, then every solution
$u \in \cone$ of (\ref{eq:main_equation_a=1}) satisfies $\|u\|_\infty
\le K_\infty$.
\begin{proposition}\label{proposition:example}
Let $\delta \in (0,\lambda_2^{rad})$ and $M>0$. Assume $f$ satisfies $(f1)$--$(f3)$ and 
(\ref{eq:M-delta-estimate}) with $a_{0}=1$. If $f'(s)< \lambda_2^{rad}$ for every $s\in [0,K_\infty]$,
then (\ref{eq:main_equation_a=1}) only admits constant solutions in $\spc$.
\end{proposition}

\begin{proof}
 Let $u\in \spc$ be a solution of \eqref{eq:main_equation_a=1}. We
 can write $u=v+\l$ for some $\l\in\R$ and $v \in \spc$ satisfying 
$$
\int_{B}v\,dx
 =0\qquad \text{and}\qquad \l_2^{rad}\int_B v^2\, dx \leq \int_B \left(|\nabla v|^2+v^2\right)\, dx.
$$
Multiplying (\ref{eq:main_equation_a=1}) by $v$ and integrating
 by parts, we obtain 
 \begin{align*}
\l_2^{rad}\int_B v^2\, dx &\le \int_B (|\nabla v|^2\,+v^2) dx =\int_B f(v+\lambda)v\, dx\\ 
&=\int_B [f(v+\lambda)-f(\lambda)]v\, dx = \int_B f'(\lambda+ cv) v^2\,dx,
 \end{align*}
with some function $c=c(x)$ satisfying $0\leq c\leq1$ in $B$. Now, since $\|u\|_{L^\infty(B)}\leq K_\infty$, we also have $\|\lambda+cv\|_{L^\infty(B)}\leq K_\infty$, hence $f'(\l
+cv)<\l_2^{rad}$ by assumption. This yields $v=0$.
\end{proof}

\begin{remark}
\label{sec:exist-solut-via-3}
If, in addition to the assumptions of Proposition~$\ref{proposition:example}$, $f$
only has one positive fixed point, then this fixed point is the only
radial solution of (\ref{eq:main_equation_a=1}). 
This is true e.g. if $f$ is given as $f(u)=g(u)u$ with a strictly
increasing $C^1$-function $g:[0,\infty) \to \R$ satisfying $g(0)=0$
and $\lim \limits_{t \to \infty}g(t) \in (1,\lambda_2^{rad})$.
\end{remark} 

In the remainder of this section we will prove Theorem~\ref{theorem:a=1}. For this reason in the following we will assume that $a(r)\equiv1$ and we always assume ($f1$)--($f4$) (with $a_0=1$). 
As we already mentioned in the introduction, we shall find a solution of \eqref{eq:main_equation_a=1} by a minimax technique. This will allow us to prove that it is nonconstant through an energy comparison. The first step is to consider a truncated problem which
can be cast into a variational setting in $H^1(B)$. We will then recover the original problem through the a priori bounds on the solutions proved in the previous section. 

\begin{lemma}\label{lemma:truncated_function}
There exist $p>1$ satisfying $p<\frac{N+2}{N-2}$ if $N \ge 3$ and a function $\tilde f$ satisfying ($f1$)-($f4$) and
\begin{equation}
  \label{eq:superlinear-subcritical}
\lim_{s \to \infty}\frac{\tilde f(s)}{s^p}=1,  
\end{equation}
such that if $u\in\cone$ solves $-\Delta u+u=\tilde f(u)$ in $B$ with
$\partial_{\nu}u=0$ on $\partial B$, then $u$ solves
\eqref{eq:main_equation_a=1}.\\
\end{lemma}

\begin{proof}
Fix
$M,\delta>0$ such that (\ref{eq:M-delta-estimate}) holds for
$f$ with $a_0=1$, i.e. 
\begin{equation}
  \label{eq:M-delta-special}
f(s) \ge (1+\delta)s \qquad \text{for $s
\ge M$.}  
\end{equation}
 By Remark~\ref{sec:exist-solut-via-1}, there exists 
$K_\infty>0$ such that, for any nonnegative nonlinearity $\tilde
f:[0,\infty) \to \R$ satisfying $\tilde f(s) \ge (1+\delta)s$ for $s
\ge M$ and any solution $u \in \cC$ of the problem 
\begin{equation}
  \label{eq:tilde}
-\Delta u + u =\tilde f(u) \quad \text{in $B$},\qquad \partial_\nu u= 0 \quad
\text{on $\partial B$}
\end{equation}
we have $\|u\|_{L^\infty(B)} \le K_\infty$. Now fix $s_0 > \max
\{K_\infty,M\}$, and fix $p>1$ with $p<\frac{N+2}{N-2}$ if $N \ge
3$. To define the truncated function $\tilde f \in C^1([0,\infty))$
we distinguish the following cases.

\medbreak

{\em Case 1:} $f(s_0)= (1+\delta)s_0$. Then it follows from
(\ref{eq:M-delta-special}) that $f(s)$ touches the line $(1+\delta)s$
from above at $s_0$, so that the two curves are tangent at $s_0$.
Therefore $f'(s_0)=1+\delta$ and we set 
$$
\tilde f(s)= 
\left \{
  \begin{aligned}
  &f(s) \qquad \text{for $0 \le s \le s_0$};\\
  &f(s_0) + f'(s_0)(s-s_0)+(s-s_0)^p \qquad \text{for $s>s_0$}.
  \end{aligned}
\right.
$$
Then $\tilde f \in C^1([0,\infty))$
satisfies~(\ref{eq:superlinear-subcritical}), and it  
also satisfies (\ref{eq:M-delta-special}), so that every
solution of (\ref{eq:tilde}) is also a solution of
(\ref{eq:main_equation_a=1}) by the choice of $K_\infty$ and $s_0$.

\medbreak

{\em Case 2:} $f(s_0)> (1+\delta)s_0$. Then we may first modify $f$ in
a right neighborhood $(s_0,s_0+\eps)$ of $s_0$, in such a way that
$f(s)\ge (1+\delta)s$ for $s \le s_0+\eps$ and $f'(s_0+\eps)=
1+\delta$. Then we define $\tilde f$ as in Case 1 with $s_0$ replaced
by $s_0+\eps$.
\end{proof}

In the following, we may also assume that $\tilde f$ is defined on the
whole real line by setting $\tilde f \equiv 0$ on $(-\infty,0]$. It
then follows by standard arguments from the subcritical growth
assumption (\ref{eq:superlinear-subcritical}) that the functional $I:\spc\to \R$ defined by 
\[
u \mapsto I(u)=\int_B\left(\frac{|\nabla u|^2+u^2}{2}-\tilde F(u)\right)\,dx,
\]
where $\tilde F(s):= \int_0^s \tilde f(t)\,dt$ is well defined and of class $C^2$ in $H^1(B)$. Moreover, critical points of $I$ are radial solutions of \eqref{eq:main_equation_a=1}. We look for critical points of $I$ by applying a mountain pass type argument in a suitable subset of $\cone$, which is based on invariance properties of the corresponding flow. 

Since the truncated nonlinearity $\tilde f$ has a subcritical growth at infinity, the Palais-Smale condition holds. We include a proof for completeness though this is a standard fact. 

\begin{lemma}\label{lemma:palais_smale}
The action functional $I$ satisfies the Palais-Smale condition.
\end{lemma}

\begin{proof}
Let $(u_n)_n \subset H^1_{rad}(B)$ be a sequence with $I'(u_n) \to 0$ and
such that $I(u_n)$ remains bounded. It easily follows from \eqref{eq:superlinear-subcritical} there exists $R_0>0$ and $\mu \in
(2,p+1)$ such that $\tilde f(s)s \ge \mu \tilde F(s)$ for $s \ge
R_0$. Hence we have 
\[
I(u_n)-\frac{1}{\mu}I'(u_n)u_n \geq \left(\frac{1}{2}-\frac{1}{\mu}\right)\|u_n\|_{H^1(B)}^2 + 
\int_{\{u_n\leq R_0\}} \left(\frac{\tilde f(u_n) u_n}{\mu}-\tilde F(u_n)\right)\,dx.
\]
Since $\mu>2$, the $H^1$-norm of the sequence $\{u_n\}$ is bounded,
hence $u_n\rightharpoonup u$ weakly in $H^1(B)$ after passing to a subsequence, where $u$ also is a
critical point of $I$. Using the subcritical
growth of $f$ given by (\ref{eq:superlinear-subcritical}) and the
compact embedding $H^1(B) \hookrightarrow L^p(B)$, it is then easy
to see that $\tilde f(u_n) \to \tilde f(u)$ strongly in the dual space $[H^1(B)]'$ of 
$H^1(B)$, and therefore -- regarding $-\Delta +Id$ as an isomorphism
$H^1(B) \to [H^1(B)]'$ -- we have
$$
u_n= [-\Delta+Id]^{-1}\tilde f(u_n) \to  [-\Delta+Id]^{-1}\tilde f(u)=u \qquad
\text{in $H^1(B)$,}
$$
as required.
\end{proof}

By assumption $(f4)$, we may now fix $u_0 \in (0,\infty)$ with
$f(u_0)=u_0$ and $f'(u_0) >\lambda_2^{rad}$. Moreover, since $u_0 <
K_\infty$, it follows from the proof
of Lemma \ref{lemma:truncated_function} that $\tilde f(u_0)= f(u_0)=u_0$
and $\tilde f'(u_0)=f'(u_0) >\lambda_2^{rad}$.
Since $\lambda_2^{rad}>1$, $u_0$ is an isolated fixed
point of $\tilde f$, so we can define 
$$
u_-:= \sup \{t \in [0,u_0)\,:\, \tilde f(t)=t\}
$$
and 
$$
u_+:= \inf \{t > u_0 \,:\, \tilde f(t)=t\}.
$$
We point out that $u_+= \infty$ is possible. Next, we define the convex set  
$$
\cC_*:= \{u \in \cC\::\: \text{$u_- \le  u \le u_+$ a.e. in $B$}\}.
$$
Clearly, $\cC_*$ is closed and convex. Moreover we have

\begin{lemma}
\label{lemma:cone_invariant_under_gradient_flow}
Fix $c\in\R$ and assume that there exist $\eps, \delta>0$ such that $\|\nabla I(u)\|_{H^1(B)}\geq\delta$ for every $u\in \cC_*$ with
$|I(u)-c|\leq2\eps$. Then there exists $\eta:\cone_*\to\cone_*$ continuous with respect to the $H^1$-topology which satisfies the following properties
\begin{itemize}
\item[(i)] $I(\eta(u))\leq I(u)$ for every $u\in\cC_*$;
\item[(ii)] $I(\eta(u))\leq c-\eps$ if $|I(u)-c|<\eps$;
\item[(iii)] $\eta(u)=u$ if $|I(u)-c|>2\eps$.
\end{itemize}
\end{lemma}

\begin{proof}
We first show that the operator $T$ defined in
(\ref{eq:definition_of_T}) -- with $a(r) \equiv 1$, $b(r) \equiv 0$
and $\tilde f$ in place of $f$ 
-- satisfies 
\begin{equation}
  \label{eq:8}
T(\cC_*) \subset \cC_*.  
\end{equation}
Let $w \in \cC_*$ and denote by $v \in H^1(B)$ the unique solution of  
\[
\left\{\begin{array}{ll}
        -\Delta v+v=\tilde f(w) \quad &\text{ in } B \\
	\partial_\nu v=0 &\text{ on } \partial B,
       \end{array}
\right.
\]
Then $v \in \cC$ by Lemma~\ref{lemma:T_is_cone_map}, so we only have
to prove that $u_- \le v \le u_+$ a.e. in $B$. Note that $h= v-u_-$ satisfies  
$$
-\Delta h+h=\tilde f(w)-u_- \ge 0 \quad \text{in $B$}\qquad
\text{and}\qquad \partial_\nu h= 0 \quad \text{on $\partial B$.}
$$
Here we used the fact that $\tilde f$ is nondecreasing and $\tilde f(u_-)=u_-$. 
Multiplying this equation with $h^-$ and integrating by parts, we
 obtain $\|h^-\|_{H^1}^2 \le 0$ and therefore $h^- \equiv 0$, i.e. $v
\ge u_-$ a.e. in $B$. Very similarly, if $u_+< \infty$, we show that $v \le
u_+$ a.e. in $B$. Hence we conclude that $v \in \cC_*$ and
(\ref{eq:8}) follows.

\medbreak 

Next, we take a smooth cut-off function $\chi:\R\to[0,1]$ such that $\chi(s)=1$ if $|s-c|<\eps$ and $\chi(s)=0$ if $|s-c|>2\eps$. For $u\in H^1(B)$
consider the following Cauchy problem
\begin{equation}
\label{eq:cauchy_problem_for_the_flux}
\left\{ \begin{array}{ll}
         \frac{d}{dt} \eta(t,u)=-\chi(I(\eta(t,u))) \frac{\nabla I(\eta(t,u))}{\|\nabla I(\eta(t,u)) \|_{H^1(B)}} \quad & t>0\\
	 \partial_\nu \eta(t,u)(x)=0 & t>0, \  x\in\partial B \\
	 \eta(0,u)=u.
        \end{array}
\right.
\end{equation}
Since $I\in C^2(H^1(B), \R)$, the normalized gradient vector field appearing in \eqref{eq:cauchy_problem_for_the_flux} is locally Lipschitz
continuous and globally bounded, hence there exists of a unique solution $\eta(\cdot,u)\in C^1([0,+\infty), H^1(B))$. We set
\begin{equation}
\label{deformation}  
\eta(u):=\eta\left(\frac{2\eps}{\delta},u\right).
\end{equation}
Properties (i), (ii) and (iii) are standard, so it remains to prove
that $\eta(\cC_*) \subset \cC_*$. To this aim we consider the approximation of
the flow line $t \mapsto \eta(t,u)$ given by the Euler polygonal. 
The first segment of the polygonal is given by the
expression
\[
\bar \eta(t,u)=u-\frac{t}{\lambda}\nabla I(u)= u-\frac{t}{\lambda}(u-T(u)), \qquad t\in(0,1),
\]
where $\lambda=\frac{\chi(I(\eta(t,u)))}{\|\nabla I(\eta(t,u)) \|_{H^1(B)}}$ and $T$ is the operator defined in \eqref{eq:definition_of_T} (with
$a(r)=1$). By writing
\[
\bar \eta(t,u)=\left(1-\frac{t}{\lambda}\right) u + \frac{t}{\lambda} T(u), \qquad t\in(0,1),
\]
we see that it is contained in $\cC_*$ by (\ref{eq:8})
 and the convexity 
$\cC_*$. Finally, since the vector field in 
\eqref{eq:cauchy_problem_for_the_flux} is locally Lipschitz, the Euler
polygonals are known to converge in $H^1(B)$ to the flow line $t
\mapsto \eta(t,u)$, which therefore
must be contained in $\cC_*$.
\end{proof}

\begin{lemma}
\label{sec:case-multiple-fixed-2}
Let $\tau>0$ be such that $\tau <\min \{u_0-u_-,u_+-u_0\}$. 
Then there exists $\alpha>0$ such that 
\begin{itemize}
\item[(i)] $I(u)\ge I(u_-)+\alpha$ for every $u \in
\cC_*$ with $\|u-u_-\|_{L^\infty(B)}=\tau$.\\
\item[(ii)] if $u_+< \infty$, then $I(u)\ge I(u_+)+\alpha$ for every $u \in
\cC_*$ with $\|u-u_+\|_{L^\infty(B)}= \tau$.
\end{itemize}

\end{lemma}
\begin{proof}
Suppose by contradiction that there exists a sequence $(w_n)_n 
\subset \cC$ of increasing nonnegative functions such that
$\|w_n\|_{L^\infty(B)}=w_n(1)=\tau$ for all $n$ and $\limsup \limits_{n \to \infty}
\bigl[I(u_-+w_n)-I(u_-)\bigr] \le 0$. Since
\begin{align*}
I(u_-+w_n)-I(u_-) &= \frac{1}{2}\int_B\left(|\nabla w_n|^2+ |w_n|^2 +2u_- w_n\right)\,dx 
 - \int_B \Bigl(\tilde F(u_- + w_n)-\tilde F(u_-))\,dx\\
&= \frac{1}{2}\int_B|\nabla w_n|^2\,dx + 
\int_B \int_0^1 \Bigl(u_-+t w_n -  \tilde f(u_-+t w_n)\Bigr)w_n\,dt dx 
\end{align*} 
and 
\begin{equation}
  \label{eq:u_0-u_-}
s-\tilde f(s)>0 \qquad \text{for $s \in (u_-,u_0)$,}  
\end{equation}
we then conclude that $\|\nabla w_n\|_{L^2}(B) \to 0$ as $n \to \infty$. Hence the sequence
$w_n$ converges to the constant solution $w \equiv \tau$ in the
$H^1$-norm. By Lemma~\ref{sec:cone-nonn-nond} we therefore
conclude that 
\begin{align*}
0 \ge \lim_{n\to \infty} \bigl[I(u_-+w_n)-I(u_-)\bigr] &= \lim_{n \to
  \infty} \int_B \int_0^1 \Bigl(u_-+t w_n -  \tilde f(u_-+t
w_n)\Bigr)w_n\,dx\\
&=  \int_B \int_0^1 \Bigl(u_-+t \tau -  \tilde f(u_-+t
\tau)\Bigr)\tau\,dt dx.
\end{align*}
This however contradicts (\ref{eq:u_0-u_-}). Hence there exists $\alpha_1>0$ such that (i) holds. 

\medbreak

In a similar way, now using the fact that $s-\tilde f(s)<0$ for $s \in (u_0,u_+)$, we find $\alpha_{2}>0$
such that (ii) holds if $u_+< \infty$. The claim then follows with $\alpha:= \min \{\alpha_1,\alpha_2\}$. 
\end{proof}

In the following, we first consider the case 
$$
u_+< \infty.
$$
Moreover,
we fix $\tau$ and $\alpha$ as in
Lemma~\ref{sec:case-multiple-fixed-2}, and we define 
\begin{equation}
  \label{eq:2}
U_\pm := \{u \in \cC_* \::\: I(u)<I(u_\pm)+\frac{\alpha}{2},\:
\|u-u_\pm \|_{L^\infty(B)} < \tau\}.
\end{equation}
Then we have:

\begin{proposition}\label{proposition:mountain_pass}
Let
\[
\Gamma=\left\{ \gamma\in C([0,1],\cC_*)\ :\  \gamma(0) \in U_-,\:
  \gamma(1) \in U_+\right\}
\]
and
\[
c=\inf_{\gamma\in\Gamma}\max_{t\in[0,1]} I(\gamma(t)).
\]
Then $c \ge \max \{I(u_-),I(u_+)\}+\alpha$ and $c$ is a critical level for $I$. More precisely, there exists a
critical point $u \in \cC_*$ of $I$ with $I(u)=c$.
\end{proposition}
\begin{proof}
It follows immediately from Lemma~\ref{sec:case-multiple-fixed-2} that 
$c> \max \{I(u_-),I(u_+)\}+\alpha$. Moreover, $\Gamma$ is nonempty,
since the path of constant functions 
\begin{equation}
  \label{eq:constant-path}
t \mapsto (1-t)u_-+t u_+  
\end{equation}
is contained in $\Gamma$. Consequently, $c< \infty$. Assume by contradiction that there
does not exist a critical point $u \in \cC_*$ of $I$ with $I(u)=c$. By Lemma
\ref{lemma:palais_smale}, this implies the existence of $\eps,\delta>0$ such that
$\|\nabla I(u)\|_{H^1(B)}\geq\delta$ for all $u \in \cC_*$ satisfying
$|I(u)-c|\leq2\eps$. Without loss of generality, we may assume that
$4\eps< \alpha$. Correspondingly, let $\eta$ be the deformation defined in Lemma
\ref{lemma:cone_invariant_under_gradient_flow}, and let
$\gamma\in\Gamma$ be such that $\max_{t\in[0,1]} I(\gamma(t)) \leq c+\eps.$

Defining $\bar \gamma: [0,1] \to \cC_*$ by $\bar
\gamma(t)=\eta(\gamma(t))$, we then have 
$\bar\gamma(0)=\gamma(0)$ and $\bar\gamma(1)=\gamma(1)$ 
because of Lemma~\ref{lemma:cone_invariant_under_gradient_flow} (iii) and the fact that
$I(u_\pm) < c- 2\eps$ by our choice of $\eps$ and $\alpha$. Hence $\bar \gamma \in \Gamma$.
However, by Lemma~\ref{lemma:cone_invariant_under_gradient_flow} (i)
and (ii) we have 
\[
\max_{t\in[0,1]} I(\bar \gamma(t)) \leq c-\eps,
\]
contradicting the definition of $c$. The claim then follows.
\end{proof}

In order to show that the critical value $c$ in Proposition
\ref{proposition:mountain_pass} does not yield a constant solution of
(\ref{eq:main_equation_a=1}), it suffices to show that $c<I(u_0)$. To
show this, we will now make use of the assumption
$\tilde f'(u_0)>\lambda_2^{rad}$. The strategy is to find a curve $\gamma\in\Gamma$ such that
$\max_{t\in[0,1]} I(\gamma(t))<I(u_0)$. This is achieved by suitably perturbing
the constant path defined in (\ref{eq:constant-path}) around $u_0$,
moving in the direction of the eigenfunction associated
to $\lambda_2^{rad}$. We will need a series of lemmas. Let us start with some simple properties of the eigenfunction associated to
$\lambda_2^{rad}$.

\begin{lemma}\label{lemma:properties_of_v}
Let $v$ be an eigenfunction associated to $\lambda_2^{rad}$, that is
\[
\left\{\begin{array}{ll}
       -\Delta v +v =\lambda_2^{rad} v \quad &\text{in } B \\
       \partial_\nu v=0 &\text{on } \partial B \\
       v \text{ radial}.
       \end{array}
\right.
\]
Then $v$ is unique up to a multiplicative factor and we can chose it increasing. Moreover, $\int_B v\, dx=0$.
\end{lemma}
\begin{proof}
By writing the equation for $v$ in radial coordinates we see that it satisfies a Sturm-Liouville problem. Hence $v$ is unique up to a
multiplicative factor, it is monotone and has exactly one zero. By taking $-v$ if necessary, we can assume it is increasing. We refer to \cite{BonheureGrumiau} for the explicit form of the eigenfunctions. By integrating the equation for $v$ we deduce $(\lambda_2^{rad}-1)\int_B v \, dx=0$, and therefore $\int_B v\, dx=0$.
\end{proof}
In the following $v$ will always denote a positive eigenfunctions associated to $\lambda_2^{rad}$.
\begin{lemma}
\label{sec:exist-solut-via-2}
Consider the function 
$$
\psi: \R^2 \to \R,\qquad \psi(s,t)=I'(t(u_0+sv))(u_0+sv).
$$
There exists $\eps_1,\eps_2>0$ and a
$C^1$-function $g:(-\eps_1,\eps_1) \to (1-\eps_2,1+\eps_2)$
such that for $(s,t) \in U := (-\eps_1,\eps_1) \times
(1-\eps_2,1+\eps_2)$ we have $\psi(s,t)=0$ if and only if $t=g(s)$.\\
Moreover: 
\begin{itemize}
\item[(i)] $g(0)=1$, $g'(0)=0$;
\item[(ii)] $I(g(s)(u_0+sv))<I(u_0)$ for $s \in (-\eps_1,\eps_1)$.
\end{itemize}
\end{lemma}

\begin{proof}
Since $I$ is a $C^2$-Functional, $\psi$ is of class $C^1$ with $\psi(0,1)=0$,
$$
\frac{\partial}{\partial t}\Big|_{(0,1)}\psi(s,t) = I''(u_0)(u_0,u_0)= \int_{B} 
[1-\tilde f'(u_0)]u_0^2 \,dx <(1-\lambda_2^{rad})|B|u_0^2<0
$$
and 
$$
\frac{\partial}{\partial s}\Big|_{(0,1)}\psi(s,t) =I'(u_0)v+  I''(u_0)(u_0,v)=
[1-\tilde f'(u_0)]u_0 \int_{B} 
v\,dx =0.
$$
Thus the existence of $\eps_1,\eps_2$ and $g$, as well as property (i),
follow from the implicit function theorem. To prove (ii), we write 
$g(s)= 1+o(s)$, so that 
$$
g(s)(u_0+sv)-u_0= (g(s)-1)u_0 +g(s)s v= sv + o(s)
$$
and therefore, by Taylor expansion,
\begin{align*}
I(g(s)(u_0+sv)) -I(u_0) &= \frac{1}{2} I''(u_0) [sv +
o(s),sv+o(s)]+o(s^2)=\frac{s^2}{2} I''(u_0) (v,v)+o(s^2)\\
&=\frac{s^2}{2}\int_B\left(|\nabla v|^2+v^2-\tilde
  f'(u_0)v^2\right)\,dx+o(s^2).
\end{align*}
Since 
$$
\int_B \left(|\nabla v|^2+v^2-\tilde
  f'(u_0)v^2\right) <\int_B\left(|\nabla v|^2+v^2-\lambda_2^{rad} v^2\right)\,dx =0,
$$
property (ii) holds after making $\eps_1$, $\eps_2$ smaller if necessary. 
\end{proof}

\begin{lemma}\label{lemma:t_eps_unique_local_maximum}
Let $\tau$ be given as in Lemma~$\ref{sec:case-multiple-fixed-2}$, and
fix 
$t_-,t_+>0$ such that  
\begin{equation}
  \label{eq:3}
t_- u_0 \in U_-,\quad t_+
u_0 \in U_+ \quad \text{and}\quad    u_- < t_- u_0 < u_0 < t_+ u_0 < u_+, 
\end{equation}
where $U_\pm$ are defined in $(\ref{eq:2})$. For $s  \ge 0$ define
\begin{equation}
  \label{eq:4}
\gamma_s: [t_-,t_+]\to H^1(B) \qquad \gamma_s(t)= t(u_0+sv).
\end{equation}
Then there exists $s>0$ such that $\gamma_s(t_\pm ) \in U_\pm$, $\gamma_s(t) \in \cC_*$ for $t_- \le t \le t_+$ and 
\begin{equation}
  \label{eq:1}
\max_{t_- \le t \le t_+} I(\gamma_s(t))<I(u_0).  
\end{equation}
\end{lemma}

\begin{proof}
We first observe that the function $t \mapsto I(\gamma_0(t))$ has a
unique maximum point at $1$, since 
$$
\frac{d}{dt} I(\gamma_0(t)) =I'(t u_0)u_0=|B|(tu_0-\tilde f(tu_0))u_0 
$$
and $tu_0-\tilde f(tu_0)>0$ in $[t_-,1)$ while $tu_0-\tilde f(tu_0)<0$
in $(1,t_+]$. Consider the neighborhood $U$ of $(s,t)=(0,1)$ found in
Lemma~\ref{sec:exist-solut-via-2}. By continuity, there exists $s_0>0$ such that 
$$
I(\gamma_s(t)) <I(u_0) \qquad \text{for every $(s,t) \in [-s_0,s_0]
  \times [t_-,t_+] \setminus U$.}
$$
On the other hand, if $(s,t)$ in $U$ is such that $t$ is the global
maximum of the function $\gamma_{s}$, then 
$$
0=\frac{d}{dt}I(\gamma_s(t))= I'(t(u_0+sv))(u_0+sv)
$$
and therefore $t=g(s)$ and $I(t(u_0+sv))<I(u_0)$ by
Lemma~\ref{sec:exist-solut-via-2}. Hence (\ref{eq:1}) follows.
By (\ref{eq:3}) and since $v$ is an increasing function, we 
may choose $s \in (0,s_0)$ so small such that 
$$
\gamma_s(t_-)=t_-(u_0+sv) \in U_- \qquad \text{and} \qquad \gamma_s(t_+)=t_+(u_0+sv) \in U_+ 
$$
By convexity, we then also have $\gamma_s(t) \in \cC_*$ for all $t \in
[t_-,t_+]$.
\end{proof}

\begin{proof}[End of the proof of Theorem \ref{theorem:a=1} in the
  case $u_+<\infty$]
Proposition \ref{proposition:mountain_pass} provides in $\cone$ a mountain pass type critical point of $I$ which, by Lemma
\ref{lemma:truncated_function}, is a solution of
\eqref{eq:main_equation_a=1}. As emphasized before, it only remains to
prove that $c<I(u_0)$, which implies that the critical point found in
Proposition~\ref{proposition:mountain_pass} is not constant. To this
end, we note that Lemma~\ref{lemma:t_eps_unique_local_maximum} implies
that -- after an affine transformation of the independent variable -- 
the path $\gamma_s$ defined in (\ref{eq:4}) belongs to $\Gamma$ and
satisfies $\max_{t}
I(\gamma_s(t))<I(u_0)$ for some $s>0$. Hence $c<I(u_0)$, as claimed.
\end{proof}

Now we consider the case 
$$
u_+=\infty.
$$
We then fix $\tau$ and $\alpha$ as in
Lemma~\ref{sec:case-multiple-fixed-2} (i), and we keep the definition of
$U_-$ from \eqref{eq:2}. In addition, we now set  
\begin{equation}
  \label{eq:5}
U_+:= \{u \in \cC_*\::\: u \ge u_0 , I(u) \le I(u_-)\}.
\end{equation}
Then we have

\begin{proposition}\label{proposition:mountain_pass-1}
Let $\Gamma$ and $c$ be defined as in
Lemma~$\ref{proposition:mountain_pass}$ (with $U_+$ now defined as in
$(\ref{eq:5})$). Then $c> I(u_-)+\alpha$, and there exists a
critical point $u \in \cC_*$ of $I$ with $I(u)=c$.
\end{proposition}
\begin{proof}
It follows from Lemma~\ref{sec:case-multiple-fixed-2} (i) that 
$c> I(u_-)+\alpha$. Moreover, considering again $M,\delta>0$ such that
(\ref{eq:M-delta-estimate}) holds, we find that, for $t>M$, 
\begin{align*}
I(t \cdot 1)&=|B|\left(\frac{t^2}{2}-\tilde F(t)\right)=
|B|\Bigl(\frac{t^2}{2} - \int_0^t \tilde f(s)\,ds \Bigr)\\
&\le |B|\Bigl(\frac{t^2}{2} - \int_0^M \tilde f(s)\,ds - (1+\delta) \int_M^t
s\,ds \Bigr)\\  
&= 
\frac{|B|}{2} \Bigl(t^2 - 2 \int_0^M \tilde f(s)\,ds - (1+\delta)(t-M)^2 \Bigr) \to -\infty
\end{align*}
as $t \to \infty$. Hence, for $\Lambda>0$ sufficiently large, the path 
$[0,1] \to \cC_*$, $t \mapsto u_- +\Lambda t$ of constant functions is contained in $\Gamma$. Consequently, $c< \infty$. Assume by contradiction that there
does not exist a critical point $u \in \cC_*$ of $I$ with $I(u)=c$. By
Lemma~\ref{lemma:palais_smale}, this implies the existence of $\eps,\delta>0$ such that
$\|\nabla I(u)\|_{H^1(B)}\geq\delta$ for all $u \in \cC_*$ satisfying
$|I(u)-c|\leq2\eps$. Without loss of generality, we may assume that
$4\eps< \alpha$. Correspondingly, let $\eta$ be the deformation defined in Lemma
\ref{lemma:cone_invariant_under_gradient_flow}, and let
$\gamma\in\Gamma$ be such that $\max \limits_{t\in[0,1]} I(\gamma(t)) \leq c+\eps.$

Defining $\bar \gamma: [0,1] \to \cC_*$ by $\bar
\gamma(t)=\eta(\gamma(t))$, we then have 
$\bar\gamma(0)=\gamma(0)$ and $\bar\gamma(1)=\gamma(1)$ 
because of Lemma
\ref{lemma:cone_invariant_under_gradient_flow} (iii) and the fact that
$I(u_\pm) < c- 2\eps$ by our choice of $\eps$ and $\alpha$. Hence $\bar \gamma \in \Gamma$.
However, by Lemma \ref{lemma:cone_invariant_under_gradient_flow} (i)
and (ii) we have 
\[
\max_{t\in[0,1]} I(\bar \gamma(t)) \leq c-\eps,
\]
contradicting the definition of $c$. The claim then follows.
\end{proof}

Again we need to show $c<I(u_0)$ for the critical value $c$ in Proposition
\ref{proposition:mountain_pass}.

\begin{lemma}\label{lemma:t_eps_unique_local_maximum-1}
Let $\tau$ be given as in Lemma~$\ref{sec:case-multiple-fixed-2}$, and
fix 
$t_-,t_+$ such that  
$$
t_- u_0 \in U_-,\quad t_+ u_0 \in U_+ \quad \text{and}\quad    u_- < t_- u_0 < u_0 < t_+ u_0 < \infty, 
$$
where $U_-$ is defined in $(\ref{eq:2})$ and $U_+$ is defined in $(\ref{eq:5})$. For $s  \ge 0$ define
\begin{equation}
  \label{eq:4-1}
\gamma_s: [t_-,t_+]\to H^1(B) \qquad \gamma_s(t)= t(u_0+sv).
\end{equation}
Then there exists $s>0$ such that $\gamma_s(t_\pm ) \in U_\pm$,
$\gamma_s(t) \in \cC_*$ for $t_- \le t \le t_+$ and $\max \limits_{t_- \le t \le t_+} I(\gamma_s(t))<I(u_0).$
\end{lemma}

\begin{proof}
The proof is exactly the same as the one of
Lemma~\ref{lemma:t_eps_unique_local_maximum} (using
Lemma~\ref{sec:exist-solut-via-2}). The only difference here is the
new definition of $U_+$.
\end{proof}

\begin{proof}[End of the proof of Theorem \ref{theorem:a=1} in the
  case $u_+= \infty$]
Lemma~\ref{lemma:t_eps_unique_local_maximum-1} implies
that -- after an affine transformation of the independent variable -- 
the path $\gamma_s$ defined in (\ref{eq:4-1}) belongs to $\Gamma$ and
satisfies $\max_{t} I(\gamma_s(t))<I(u_0)$, so that
$c<I(u_0)$. Hence the mountain pass type critical point of $I$ in
$\cC_*$ provided by Proposition \ref{proposition:mountain_pass} 
is not constant.
\end{proof}

We conclude with the remark that the method presented in this section also applies to obtain decreasing solutions in the subcritical regime assuming for instance the standard Ambrosetti-Rabinowitz condition.

\begin{remark}
Let $a(|x|) \in C(\overline{B})$ be nonincreasing and strictly positive. Let $f$ satisfy ($f1$), ($f2$) and assume moreover that
\[
\text{there exist } C>0, 2<p<2^* \text{ such that } |f(s)|\leq C |s|^{p-1},
\]
\[
\text{there exist } R_0>0, \mu>2 \text{ such that } \ f(s)s\geq\mu F(s) \ \text{for every } s\geq R_0.
\]
Then the following holds
\begin{itemize}
\item[(i)] if $a(|x|)$ is nonconstant, there exists at least one decreasing radial solution of \eqref{eq:main_equation}.
\item[(ii)] if $a(|x|)=1$ and ($f4$) holds then \eqref{eq:main_equation_a=1} admits both an increasing and a decreasing radial solution, which are
not constant.
\end{itemize}
\end{remark}


\def\cprime{$'$} \def\cprime{$'$}

\end{document}